\documentclass[10pt,letterpaper,leqno,pdftex]{amsart}
\usepackage[all]{xy}
\usepackage[mathscr]{euscript}
\usepackage{amsmath, amsfonts, amssymb, mathrsfs, stmaryrd}
\usepackage{appendix}
\usepackage{amsthm}
\usepackage{comment}
\usepackage{bm}
\usepackage{lipsum}
\newcommand{\dc}{\backslash}
\newcommand{\sm}{\smallsetminus}
\newcommand{\mapp}[1]{\xrightarrow{#1}}
\newcommand{\Oo}{\mathcal{O}}

\newcommand{\CC}{\mathbb{C}}

\newcommand{\FF}{\mathbb{F}}
\newcommand{\ZZ}{\mathbb{Z}}
\newcommand{\QQ}{\mathbb{Q}}
\newcommand{\MM}{\text{\upshape M}}
\newcommand{\NN}{\text{\upshape N}}

\newcommand{\lqq}{\leqslant}
\newcommand{\gqq}{\geqslant}
\newcommand{\YYY}{\mathscr{Y}}
\newcommand{\ZZZ}{\mathscr{Z}}

\newcommand{\MMM}{\mathscr{M}}

\newcommand{\OOO}{\mathscr{O}}

\newcommand{\XXX}{\mathscr{X}}
\newcommand{\CCC}{\mathscr{C}}
\newcommand{\WWW}{\mathscr{W}}

\DeclareMathOperator{\length}{length}

\DeclareMathOperator{\Diff}{Diff}
\DeclareMathOperator{\id}{id}
\DeclareMathOperator{\Tr}{Tr}

\DeclareMathOperator{\Cl}{Cl}

\DeclareMathOperator{\Aut}{Aut}

\DeclareMathOperator{\End}{End}
\DeclareMathOperator{\Gal}{Gal}
\DeclareMathOperator{\Hom}{Hom}
\DeclareMathOperator{\disc}{disc}

\DeclareMathOperator{\Spec}{Spec}
\DeclareMathOperator{\Nrd}{Nrd}

\DeclareMathOperator{\charr}{char}

\DeclareMathOperator{\Lie}{Lie}
\DeclareMathOperator{\inv}{inv}
\DeclareMathOperator{\diag}{diag}
\DeclareMathOperator{\Def}{Def}

\newcommand{\frakD}{\mathfrak D}

\newcommand{\frakm}{\mathfrak m}

\newcommand{\fraka}{\mathfrak a}
\newcommand{\frakg}{\mathfrak g}
\newcommand{\frakG}{\mathfrak G}

\newcommand{\frakp}{\mathfrak p}
\newcommand{\frakq}{\mathfrak q}

\newcommand{\ord}{\mathrm{ord}}

\theoremstyle{plain}
\newtheorem{theorem}{Theorem}[section]
\newtheorem{corollary}[theorem]{Corollary}

\newtheorem{proposition}[theorem]{Proposition}

\theoremstyle{definition}
\newtheorem{definition}[theorem]{Definition}

\numberwithin{equation}{section}

\theoremstyle{plain}
\newtheorem*{theorem*}{Theorem}

\theoremstyle{plain}
\newtheorem*{theorem 1}{Theorem 1}

\theoremstyle{plain}
\newtheorem*{theorem 2}{Theorem 2}

\theoremstyle{plain}
\newtheorem*{theorem 3}{Theorem 3}

\theoremstyle{plain}
\newtheorem*{theorem 4}{Theorem 4}

\begin{document}
\title{Special endomorphisms of QM abelian surfaces}
\author{Andrew Phillips}
\maketitle

\begin{abstract}
In this paper we generalize a theorem of Kudla-Rapoport-Yang which gives a formula for the arithmetic
degree of the moduli space of CM elliptic curves together with a special endomorphism of a specified degree.
Our extension is to the moduli space of QM abelian surfaces with CM together with a special endomorphism
of a specified QM degree.
\end{abstract}

\section{Introduction}

Let $K$ be an imaginary quadratic
field with discriminant $d_K$, let $s$ be the number of distinct prime factors of $d_K$, and write 
$x \mapsto \overline{x}$ for the nontrivial element of $\Gal(K/\QQ)$. Let $e_p$ and $f_p$ be the ramification index and residue field degree of $K/\QQ$ at a prime $p$.

\subsection{Elliptic curves}
Let $\ZZZ$ be the algebraic stack (in the sense of \cite{Vistoli}) over $\Spec(\Oo_K)$ with fiber $\ZZZ(S)$
the category of pairs $(E, \kappa)$ where $E$ is an elliptic curve over the $\Oo_K$-scheme $S$ and
$\kappa : \Oo_K \to \End_S(E)$ is an action such that the induced map $\Oo_K \to \End_{\OOO_S}(\Lie(E)) \cong
\OOO_S(S)$ is the structure map. A \textit{special endomorphism} of an object $(E, \kappa)$ of $\ZZZ(S)$ is an
endomorphism $f \in \End_S(E)$ satisfying
$$
\kappa(x) \circ f = f \circ \kappa(\overline{x})
$$
for all $x \in \Oo_K$. For any positive integer $m$ let $\ZZZ_m$ be the algebraic stack over $\Spec(\Oo_K)$
with $\ZZZ_m(S)$ the category of triples $(E, \kappa, f)$ where $(E, \kappa)$ is an object of $\ZZZ(S)$
and $f \in \End_S(E)$ is a special endomorphism satisfying $\deg(f) = m$ on every connected component of $S$.
Define the \textit{arithmetic degree} of $\ZZZ_m$ to be
\begin{equation}\label{arithmetic degree II}
\deg(\ZZZ_m) = \sum_{\frakp \subset \Oo_K}\log(|\FF_\frakp|)\sum_{z \in [\ZZZ_m(\overline{\FF}_\frakp)]}
\text{length}(\OOO^{\text{sh}}_{\ZZZ_m, z}),
\end{equation}
where $[\ZZZ_m(\overline{\FF}_\frakp)]$ is the set of isomorphism classes of objects in $\ZZZ_m(\overline{\FF}_\frakp)$
and $\OOO^{\text{sh}}_{\ZZZ_m, z}$ is the strictly Henselian local ring of $\ZZZ_m$ at $z$. Also,
the outer sum is over all prime ideals $\frakp \subset \Oo_K$ and $\FF_\frakp = \Oo_K/\frakp$.

For each $m \in \ZZ^+$ define a nonempty finite set of prime numbers
$$
\Diff(m) = \{\ell < \infty : (d_K, -m)_\ell = -1\},
$$
where $(\cdot\hspace{.5mm}, \cdot)_\ell$ is the usual Hilbert symbol. For any positive integer $m$
let $\rho(m)$ be the number of ideals in $\Oo_K$ of norm $m$. For any prime $\ell$ let $\rho_\ell(m)$ be the number
of ideals in $\Oo_{K, \ell}$ of norm $m\ZZ_\ell$, so there is a product formula
$$
\rho(m) = \prod_\ell \rho_\ell(m).
$$
The following is \cite[Theorem 5.15]{KRY}.

\begin{theorem 1}[Kudla-Rapoport-Yang]
Let $m \in \ZZ^+$, suppose $\Diff(m) = \{p\}$ for some prime $p$, and assume $-d_K$ is prime. The stack $\ZZZ_m$ is of dimension zero, it is supported in characteristic $p$, and
$$
\deg(\ZZZ_m) = 2\log(p)\cdot \rho(mp^{e_p-2})\cdot(\ord_p(m) + 1).
$$
If $\#\Diff(m) > 1$ then $\deg(\ZZZ_m) = 0$.
\end{theorem 1}

\subsection{QM abelian surfaces}
Let $B$ be an indefinite quaternion algebra over $\QQ$, let
$\Oo_B$ be a maximal order of $B$, and let $d_B$ be the discriminant of $B$. 
We assume each prime dividing $d_B$ is inert in $K$, so in particular, $K$ splits $B$. Let $\YYY$ be
the algebraic stack over $\Spec(\Oo_K)$ with $\YYY(S)$ the category of triples $(A, i, \kappa)$ where $A$ is an abelian scheme of relative dimension $2$ over the $\Oo_K$-scheme $S$ and  
$$
i : \Oo_B \to \End_S(A), \quad
\kappa : \Oo_K \to \End_{\Oo_B}(A)
$$ 
are injective ring homomorphisms. In particular, $1 \mapsto \id_A$ in each of these maps and they have commuting images in $\End_S(A)$. The pair $(A, i)$ is called a \textit{QM abelian surface} and we call the triple $(A, i, \kappa)$ a \textit{CMQM abelian surface}.
Our convention is that the induced map
$\Oo_K \to \End_{\Oo_B}(\Lie(A))$ is through the structure map $\Oo_K \to \OOO_S(S)$ for any object
$A$ of $\YYY(S)$ (see \cite[\S 3]{P} for the basic theory of CMQM abelian surfaces). 
A \textit{special endomorphism} of an
object $(A, i, \kappa)$ of $\YYY(S)$ is an endomorphism $f \in \End_{\Oo_B}(A)$ satisfying
$$
\kappa(x) \circ f = f \circ \kappa(\overline{x})
$$
for all $x \in \Oo_K$. Note that this definition of special endomorphism is a bit more restrictive than what is traditionally called a ``special endomorphism" of a general QM abelian surface, for example in \cite[\S 3.4]{KRY}. See Section 1.4 for a comparison.

Any nonzero $f \in \End_{\Oo_B}(A)$ is an isogeny and over any connected $S$, there is a positive definite quadratic form $\deg^\ast : \End_{\Oo_B}(A) \to \ZZ$, which we call the \textit{QM degree}, satisfying $\deg^\ast(f)^2 = \deg(f)$ (see \cite[\S 2]{P}). 
For any positive integer $m$ let $\YYY_m$ be the algebraic stack over
$\Spec(\Oo_K)$ with $\YYY_m(S)$ the category of quadruples $(A, i, \kappa, f)$ where $(A, i, \kappa)$ is an object
of $\YYY(S)$ and $f \in \End_{\Oo_B}(A)$ is a special endomorphism satisfying $\deg^\ast(f) = m$
on every connected component of $S$. Define the \textit{arithmetic degree} of $\YYY_m$ just as in 
(\ref{arithmetic degree II}). For each $m \in \ZZ^+$ define a nonempty finite set of prime numbers
$$
\Diff_B(m) = \{\ell < \infty : (d_K, -m)_\ell\cdot\inv_\ell(B) = -1\},
$$
where $\inv_\ell(B)$ is the local invariant of $B$ at $\ell$ (it is $-1$ if $B$ is ramified at $\ell$ and $1$ otherwise).
For any prime $p$ set $\varepsilon_p = 1 - \ord_p(d_B)$ and let $r$ be the number of primes dividing $d_B$.
The following (Theorem \ref{final formula II} in the text) is our generalization of Theorem 1. This result
solves a problem posed in \cite[\S 6]{KRY}.

\begin{theorem 2}
Let $m \in \ZZ^+$ and suppose $\Diff_B(m) = \{p\}$. The stack $\YYY_m$ is of dimension zero, it is supported in
characteristic $p$, and
$$
\deg(\YYY_m) = 2^{r+s}\log(p)\cdot \rho(md_B^{-1}p^{(e_p-1)\varepsilon_p - 1})\cdot (\ord_p(md_K) + \varepsilon_pf_p - \varepsilon_p).
$$
If $\#\Diff_B(m) > 1$ then $\deg(\YYY_m) = 0$.
\end{theorem 2}

The proof of this theorem uses different ideas than those in \cite{KRY} and relies on the method developed in the proof of \cite[Theorem 2.27]{Howard}.

\subsection{Eisenstein series}
Theorem 1 is only half of the main result of \cite{KRY}, which is an equality relating $\deg(\ZZZ_m)$ with
the $m$-th Fourier coefficient of a certain Eisenstein series. We explain this result in this section.
Assume $q = -d_K$ is prime. For each place $\ell \lqq \infty$ of $\QQ$
define a character $\psi_\ell : \QQ_\ell^\times \to \{\pm 1\}$ by $\psi_\ell(x) = (x, d_K)_\ell$ 
and for any
$$
\gamma = \begin{bmatrix} a & b \\ c & d \end{bmatrix} \in \Gamma = \text{SL}_2(\ZZ)
$$
define
$$
\Phi^{-}(\gamma) = \left\{\begin{array}{ll}
\psi_q(a) & \text{if $q \mid c$} \\
-iq^{-1/2}\psi_q(c) & \text{if $q \nmid c$}.
\end{array} \right.
$$
For $\tau = u + iv$ in the complex upper half plane and $s \in \CC$ with $\text{Re}(s) > 1$ define
$$
E^\ast(\tau, s) = v^{s/2}q^{(s+1)/2}\pi^{-(s+2)/2}\Gamma\left(\frac{s+2}{2}\right)L(s, \psi_q)\sum_{\gamma \in
\Gamma_\infty\dc\Gamma}\frac{\Phi^{-}(\gamma)}{(c\tau + d)|c\tau + d|^s},
$$
where $\Gamma_\infty = \{\gamma \in \Gamma : c = 0\}$. This series has meromorphic continuation to all
$s \in \CC$ and defines a non-holomorphic modular form of weight $1$. It has a Fourier expansion
$$
E^\ast(\tau, s) = \sum_{m \in \ZZ}a_m(v, s)\cdot e^{2\pi im\tau}
$$
for some functions $a_m(v, s)$ holomorphic in a neighborhood of $s = 0$. The following is
\cite[Theorem 3]{KRY}.

\begin{theorem*}[Kudla-Rapoport-Yang]
Let $m \in \ZZ^+$ and assume $-d_K$ is prime. The derivative $a'_m = a'_m(v, 0)$ is independent of $v$ and 
$\deg(\ZZZ_m) = -a'_m$.
\end{theorem*}

Most likely there is a similar theorem for the stack $\YYY_m$, but we do not pursue that direction here.
There is a similar formula relating the degree of a modified stack $\YYY'_m$ of CMQM abelian surfaces and Fourier coefficients of a modular form, described in the next section.

\subsection{Related work}

In this section we will describe the connection between Theorem 2 above and a related, but independent, result of Kudla-Yang in \cite{KY}. To explain this, let $\MMM$ be the algebraic stack, regular and flat of relative dimension $1$ over $\Spec(\Oo_K)$, with $\MMM(S)$ the category of QM abelian surfaces over $S$ satisfying the Kottwitz condition of \cite[(1.4)]{P}. Any object of $\YYY(S)$ automatically satisfies the Kottwitz condition, so there is a forgetful morphism $\YYY \to \MMM$. For  
$(A, i) \in \MMM(S)$, let 
$$
L(A, i) = \{f \in \End_{\Oo_B}(A) : \Tr(f) = 0\}.
$$
Elements of $L(A, i)$ are traditionally called ``special endomorphisms", but for $(A, i, \kappa) \in \YYY(S)$, we will see below that not every element of $L(A, i)$ is a special endomorphism in the sense of Section 1.2.
If $S$ is connected, $(A, i) \in \MMM(S)$, and $f \in L(A, i)$ is nonzero, then $f \circ f$ is a negative integer. Indeed, for any geometric point $\overline{s}$ of $S$, the natural map $\End_{\Oo_B}(A) \to \End_{\Oo_B}(A_{\overline{s}})$ is injective by the rigidity lemma for abelian schemes. By \cite[Corollary 2.6]{P}, the ring $\End_{\Oo_B}(A_{\overline{s}})$ is an order in $\QQ$, an imaginary quadratic field, or a definite quaternion algebra over $\QQ$, and the claim is clear in each of these cases. There is an induced quadratic form $Q : L(A, i) \to \ZZ$ defined by $Q(f)\cdot\id_A = -(f\circ f)$. For $m \in \ZZ^+$, let $\XXX(m)$ be the algebraic stack over $\Spec(\Oo_K)$ with $\XXX(m)(S)$ the category of triples $(A, i, f)$ where $(A, i) \in \MMM(S)$ and
$f \in L(A, i)$ satisfies $Q(f) = m$ on every connected component of $S$.

Let $\theta : \Oo_K \to \Oo_B$ be an embedding of rings and let $M_\theta$ be the $\ZZ$-module $\Oo_B$, viewed as a right $\Oo_B$-module via right multiplication and a left $\Oo_K$-module via $\theta$ and left multiplication. With $\ZZZ$ as in Section 1.1, there is a morphism of stacks 
$$
j_{M_\theta} : \ZZZ \to \MMM, \quad (E, \kappa_0) \mapsto (M_\theta \otimes_{\Oo_K} E, i_{M_\theta}),
$$
where $i_{M_\theta}$ is induced by the action of $\Oo_B$ on $M_\theta$. Note that $j_{M_\theta}$
actually has image in $\YYY$ since $j_{M_{\theta}}(E, \kappa_0)$ carries the $\Oo_K$-action
$\kappa(a) = \id \otimes \kappa_0(a)$. For $m \in \ZZ^+$, define 
$$
j^{\ast}_{M_\theta}(\XXX(m)) = \ZZZ \times_{j_{M_\theta}, \MMM} \XXX(m),
$$
which is a stack over $\Spec(\Oo_K)$, whose fiber over $S$ has objects $(E, \kappa_0, A, i, f)$, where
$(E, \kappa_0) \in \ZZZ(S)$, $(A, i, f) \in \XXX(m)(S)$, and $(A, i) \cong j_{M_\theta}(E, \kappa_0)$ in
$\MMM(S)$. 

Viewing $j^{\ast}_{M_\theta}(\XXX(m))$ as a divisor on $\ZZZ$ through the first projection, the main result of \cite{KY} gives a formula for $\deg(j^{\ast}_{M_\theta}(\XXX(m)))$ in terms of coefficients of a modular form of weight $\frac{3}{2}$. More precisely, for $\tau$ in the complex upper half plane and $s \in \CC$, there is an Eisenstein series $E^\ast(\tau, s; \fraka, \lambda, r)$ of weight $1$ depending on
a fractional $\Oo_K$-ideal $\fraka$, determined by the embedding $\theta$, an element $\lambda \in \frakD^{-1}\fraka/\fraka$, where $\frakD$ is the different ideal of $K/\QQ$, and an element $r \in \frakD^{-1}/\Oo_K$. Also, for $r \in \frakD^{-1}/\Oo_K$, define
$$
\Theta(\tau; r) = \sum_{\substack{\alpha \in \frakD^{-1}, \Tr_{K/\QQ}(\alpha) = 0 \\ \alpha \equiv r \bmod \Oo_K}} e^{2\pi i\tau\NN_{K/\QQ}(\alpha)}.
$$
Assuming $(d_K, d_B) = 1$ and $d_K$ is odd, by \cite[Main Theorem A, B]{KY}, $\deg(j^{\ast}_{M_\theta}(\XXX(m)))$ is the coefficient of $e^{2\pi im\tau}$ in the modular form
$$
-\frac{1}{2}\sum_{\substack{r \in \frakD^{-1}/\Oo_K \\ \Tr_{K/\QQ}(r) = 0}}\Theta(\tau; r)(E^{\ast})'(d_B\tau, 0; \fraka, \lambda', r),
$$ 
where $\lambda'$ is a certain ``twist" of $\lambda$.
Then using ideas similar to those in \cite[\S 4]{YY}, it seems likely there is an explicit formula for $\deg(j^{\ast}_{M_\theta}(\XXX(m)))$ similar to that for $\deg(\YYY_m)$ given in Theorem 2.

Now we consider the relationship between the stacks $\YYY_m$ and $j^{\ast}_{M_\theta}(\XXX(m))$.
Let $\frakm_B \subset \Oo_B$ be the unique ideal of index $d_B^2$ and for $(A, i) \in \MMM(S)$, let $A[\frakm_B] = \bigcap_{x \in \frakm_B}\ker(i(x))$ be the $\frakm_B$-torsion group scheme. If $\widetilde{\theta} : \Oo_K \to \Oo_B/\frakm_B$ is the reduction of $\theta$, define an algebraic stack $\YYY(\widetilde{\theta})$ over $\Spec(\Oo_K)$ where $\YYY(\widetilde{\theta})(S)$ is the full subcategory of $\YYY(S)$ whose objects $(A, i, \kappa)$ are such that the diagram
$$
\xymatrix{
\Oo_K \ar[rr]^>>>>>>>>>>>>>>>>{\kappa^{\frakm_B}} \ar[dr]_{\widetilde{\theta}} && \End_{\Oo_B/\frakm_B}(A[\frakm_B]) \\
& \Oo_B/\frakm_B \ar[ur]_{i^{\frakm_B}} & }
$$
commutes. By \cite[Theorem 3.19]{P}, the morphism $j_{M_\theta}$ defines an isomorphism of stacks
$\ZZZ \to \YYY(\widetilde{\theta})$.

To interpret $j^{\ast}_{M_\theta}(\XXX(m))$ as a space similar to $\YYY_m$, define $\YYY_m'$ to be the algebraic stack over $\Spec(\Oo_K)$ where $\YYY_m'(S)$ is the category of quadruples $(A, i, \kappa, f)$ with $(A, i, \kappa) \in \YYY(S)$ and $f \in L(A, i)$ satisfying $Q(f) = m$ on every connected component of $S$. Also, let $\YYY_m'(\widetilde{\theta})$ be the algebraic stack over $\Spec(\Oo_K)$ where $\YYY_m'(\widetilde{\theta})(S)$ is the full subcategory of $\YYY_m'(S)$ with 
$(A, i, \kappa) \in \YYY(\widetilde{\theta})(S)$. It follows from \cite[Theorem 3.19]{P} that the inclusion functor 
$$
\bigsqcup_{\eta : \Oo_K \to \Oo_B/\frakm_B}\YYY_m'(\eta) \to \YYY_m'
$$
is an isomorphism of stacks. Now define a functor
$$
\Phi : j^{\ast}_{M_\theta}(\XXX(m)) \to \YYY_m'(\widetilde{\theta}), \quad (E, \kappa_0, A, i, f)
\mapsto (M_\theta \otimes_{\Oo_K} E, i_{M_\theta}, \id_{M_\theta} \otimes \kappa_0, f'),
$$
where $f'$ is the endomorphism of $M_\theta \otimes_{\Oo_K} E$ induced by $f$ and the isomorphism
$A \cong M_\theta \otimes_{\Oo_K} E$. Using that $j_{M_\theta} : \ZZZ \to \YYY(\widetilde{\theta})$ is an equivalence, it is straightforward to show $\Phi$ is an equivalence, so the main result of \cite{KY} gives a formula for
$$
\deg(\YYY_m') = \sum_{\eta : \Oo_K \to \Oo_B/\frakm_B}\deg(\YYY_m'(\eta)).
$$

Finally, to relate $\YYY_m$ and $\YYY_m'$, we claim that
if $(A, i, \kappa) \in \YYY(S)$ for some $\Oo_K$-scheme $S$ and $f \in \End_{\Oo_B}(A)$ is a special endomorphism, as defined in Section 1.2, then $f \in L(A, i)$. To show this, it suffices to assume $S = \Spec(k)$ with $k$ an algebraically closed field. If $\charr(k) = 0$ or $\charr(k) > 0$ and $\End_{\Oo_B}(A)$ is an order in $K$, then $f = 0$ by Proposition \ref{special endo} below. If $\End_{\Oo_B}(A)$ is an order in a definite quaternion algebra, see the proof of Proposition \ref{diff}. 
However, there are elements of $L(A, i)$ that are not special endomorphisms in our terminology; for example, $\kappa(a)$ for any nonzero $a \in \Oo_K$ satisfying $\Tr_{K/\QQ}(a) = 0$. Also, $Q(f) = \deg^\ast(f)$ for any special endomorphism $f$ by the proof of Proposition \ref{diff}, so $\YYY_m$ is a full subcategory of $\YYY_m'$, but in general they are not equal. To see that their arithmetic degrees may be different, let
$K = \QQ(\sqrt{D})$ with $D < 0$ squarefree and choose
$m \in \ZZ^+$ such that $-m/D = a^2$ is a square in $\ZZ$. Then for $(A, i, \kappa) \in \YYY(\overline{\FF}_\frakp)$, the quadruple $(A, i, \kappa, \kappa(a\sqrt{D}))$ is an object of
$\YYY_m'(\overline{\FF}_\frakp)$ but not $\YYY_m(\overline{\FF}_\frakp)$.
The method used in this paper to compute $\deg(\YYY_m)$ relies crucially on the $\ZZ$-module of special endomorphisms of an object $(A, i, \kappa)$  of $\YYY$
being stable under the action of $\Oo_K$ on $\End_{\Oo_B}(A)$. The same is not true of the $\ZZ$-module $L(A, i)$.

\subsection{Notation}  
If $\CCC$ is a category, we write $C \in \CCC$ to mean $C$ is an object of $\CCC$.
We use $\Delta$ to denote the maximal order in the unique quaternion division algebra over $\QQ_p$ and $\frakg$ for the unique connected $p$-divisible group of height 2 and dimension 1 over $\overline{\FF}_p$, so $\Delta = \End_{\overline{\FF}_p}(\frakg)$. For any number field $L$, we write $\widehat{L} = L \otimes_{\QQ} \widehat{\QQ}$ for the ring of finite adeles over $L$ and $\Cl(\Oo_L)$ for the ideal class group of $L$. If $M$ is a $\ZZ$-module and $V$ a $\QQ$-vector space, let $\widehat{M} = M \otimes_{\ZZ} \widehat{\ZZ}$ and $\widehat{V} = 
V \otimes_{\QQ} \widehat{\QQ}$. If $A \to S$ is an abelian scheme, we write $\End_S^0(A)$ for $\End_S(A) \otimes_{\ZZ}\QQ$. We assume each prime dividing $d_B$ is inert in $K$.

\subsection{Acknowledgment}
This research forms part of my Boston College Ph.D. thesis. I would like to thank my advisor Ben Howard for helping with the work in \cite{P}, which forms the technical basis for this paper. I would also like to thank the anonymous referee for pointing out the results of Kudla-Yang in \cite{KY}, as well as Tonghai Yang for a helpful explanation of the connection between the results of \cite{KY} and those of this paper.

\section{Moduli spaces}

We continue with the same notation as in Sections 1.1-1.2. In the following definitions, we will mostly suppress the map $i : \Oo_B \to \End_S(A)$ from the notation for simplicity.

\begin{definition}
Define $\YYY$ to be the category whose objects are triples $(A, i, \kappa)$ where $(A, i)$ is a QM abelian surface over some
$\Oo_K$-scheme with complex multiplication $\kappa : \Oo_K \to \End_{\Oo_B}(A)$. In particular, if $A$ is defined over $S$, we require that the induced map $\Oo_K \to \End_{\Oo_B}(\Lie(A))$ is through the structure map $\Oo_K \to \OOO_S(S)$. 
A morphism $(A', i', \kappa') \to (A, i, \kappa)$
between two such triples defined over $\Oo_K$-schemes $T$ and $S$, respectively, is a morphism of $\Oo_K$-schemes
$T \to S$ together with an $\Oo_K$-linear isomorphism $A' \to A \times_S T$ of QM abelian surfaces.
\end{definition}

\begin{definition}
Let $(A, i, \kappa) \in \YYY(S)$ for some $\Oo_K$-scheme $S$. A \textit{special endomorphism} of $(A, \kappa)$
is an endomorphism $f \in \End_{\Oo_B}(A)$ satisfying
$$
\kappa(x) \circ f = f \circ \kappa(\overline{x})
$$
for all $x \in \Oo_K$. Let $L(A, \kappa)$ be the $\ZZ$-module of all special endomorphisms of $(A, \kappa)$ and set $V(A, \kappa) = L(A, \kappa)
\otimes_{\ZZ} \QQ$.
\end{definition}

We make $L(A, \kappa)$ into a left $\Oo_K$-module through the action $x \cdot f = \kappa(x) \circ f$. For connected $S$, there is the $\ZZ$-valued quadratic
form $\deg^\ast$ on $L(A, \kappa)$ and this satisfies
$$ 
\deg^\ast(x\cdot f) = \NN_{K/\QQ}(x)\cdot \deg^\ast(f)
$$
for all $x \in \Oo_K$ (\cite[Lemma 3.3]{P}).

\begin{definition}
For any positive integer $m$, define $\YYY_m$ to be the category whose objects are triples $(A, \kappa, f)$ where
$(A, i, \kappa) \in \YYY(S)$ for some $\Oo_K$-scheme $S$ and $f \in L(A, \kappa)$ satisfies $\deg^\ast(f) = m$
on every connected component of $S$. A morphism
$$
(A', \kappa', f') \to (A, \kappa, f)
$$
between two such triples, with $(A', i', \kappa')$ and $(A, i, \kappa)$ QM abelian surfaces with CM over $\Oo_K$-schemes $T$ and
$S$, respectively, is a morphism of $\Oo_K$-schemes $T \to S$ together with an $\Oo_K$-linear isomorphism $A' \to A \times_S T$
of QM abelian surfaces compatible with $f$ and $f'$.
\end{definition}

The categories $\YYY$ and $\YYY_m$ are algebraic stacks of finite type over $\Spec(\Oo_K)$, with $\YYY$ finite \'etale
over $\Spec(\Oo_K)$.
It is shown in \cite[\S 3]{P} that for any prime $\frakp \subset \Oo_K$, the group $W_B \times \Cl(\Oo_K)$ acts simply transitively on $[\YYY(\overline{\FF}_\frakp)]$, the set of isomorphism classes of objects of
$\YYY(\overline{\FF}_\frakp)$, where $W_B$ is the Atkin-Lehner group of $\Oo_B$, and that for any $A \in \YYY(\overline{\FF}_\frakp)$,
there is an isomorphism of CMQM abelian surfaces $A \cong M \otimes_{\Oo_K} E$ for some $\Oo_B \otimes_{\ZZ} \Oo_K$-module $M$, free of rank $4$ over $\ZZ$, and some elliptic curve $E$ over
$\overline{\FF}_\frakp$ with CM by $\Oo_K$ (supersingular in the case of the prime below $\frakp$ nonsplit in $K$). Here, and in the remainder of the paper, $\FF_\frakp = \Oo_K/\frakp$ and $\Spec(\overline{\FF}_\frakp)$ is an $\Oo_K$-scheme through the reduction map $\Oo_K \to \Oo_K/\frakp \hookrightarrow \overline{\FF}_\frakp$.

For each prime number $p$, define $B^{(p)}$ to be the quaternion division algebra over $\QQ$ determined
by 
$$
\inv_\ell(B^{(p)}) = \left\{\begin{array}{ll}
\inv_\ell(B) & \text{if $\ell \notin \{p, \infty\}$} \\
-\inv_\ell(B) & \text{if $\ell \in \{p, \infty\}$}.
\end{array} \right.
$$
If $(A, \kappa) \in  \YYY(k)$ with $k$ a field of characteristic $p$, we say $A$ is \textit{supersingular} if the underlying abelian surface $A$ is supersingular, that is, $A$ is isogenous to $E^2$ for some supersingular elliptic curve $E$ over $k$.

\begin{proposition}\label{special endo}
Let $(A, \kappa) \in \YYY(k)$ where $k$ is an algebraically closed field together with a ring homomorphism $\Oo_K \to k$. \\
{\upshape(a)} If $\charr(k) = 0$ then $V(A, \kappa) = 0$. \\
{\upshape(b)} If $\charr(k) > 0$
then 
$$
\dim_K(V(A, \kappa)) = \left\{\begin{array}{ll}
1 & \text{if $A$ is supersingular} \\
0 & \text{otherwise}.
\end{array} \right.
$$
\end{proposition}

\begin{proof}
(a) In this case, $\End_{\Oo_B}(A)$ is an order in $K$ (\cite[Corollary 2.6]{P}), so $\kappa : \Oo_K \to \End_{\Oo_B}(A)$ must be an isomorphism. It follows that
$L(A, \kappa) = 0$. 

(b) Now suppose $\charr(k) = p$. Note that $\MM_2(K) \cong B \otimes_{\QQ} K \hookrightarrow \End^0_k(A)$, so $A$ is not simple and hence $A \sim E^2$ for some elliptic curve $E$ over $k$ (\cite[Proposition 2.4]{P}). If $E$ is ordinary then $\End^0_{\Oo_B}(A) \cong K$ and $L(A, \kappa) = 0$ as above.
If $E$ is supersingular then $\End^0_k(A) \cong \MM_2(B_{p, \infty})$, where $B_{p, \infty}$ is the quaternion algebra over $\QQ$ ramified at $p$ and $\infty$. Hence
$$
16 = \dim_{\QQ}(\MM_2(B_{p, \infty})) = \dim_{\QQ}(B)\dim_{\QQ}(\End^0_{\Oo_B}(A))
$$
by \cite[Proposition 7.7.8]{Voight}, which implies $\End^0_{\Oo_B}(A) \cong B^{(p)}$ by \cite[Corollary 2.6]{P}. As $K$ is a simple $\QQ$-algebra
and $B^{(p)}$ is a central simple $\QQ$-algebra, by the Skolem-Noether theorem applied to the two maps
$K \to B^{(p)}$ given by $x \mapsto \kappa(x)$ and $x \mapsto \kappa(\overline{x})$, there is an $f \in (B^{(p)})^\times$
such that $\kappa(x) = f \circ \kappa(\overline{x}) \circ f^{-1}$ for all $x \in K$. This means $f \in V(A, \kappa)$, so
$\dim_K(V(A, \kappa)) \gqq 1$. However, the $K$-subspaces $\kappa(K)$ and $V(A, \kappa)$ in $B^{(p)}$ intersect trivially, so
$B^{(p)} = \kappa(K) \oplus V(A, \kappa)$ and $\dim_K(V(A, \kappa)) = 1$.
\end{proof}

For each place $\ell \lqq \infty$ of $\QQ$ let $(\cdot\hspace{.5mm}, \cdot)_\ell : \QQ_\ell^\times \times \QQ_\ell^\times \to
\{\pm 1\}$ be the Hilbert symbol. For each positive integer $m$ define a finite set of prime numbers
$$
\Diff_B(m) = \{\ell < \infty : (d_K, -m)_\ell \cdot \inv_\ell(B) = -1\}.
$$
From the product formula
$$
\prod_{\ell \lqq \infty}(d_K, -m)_\ell\cdot\inv_\ell(B) = 1
$$
and $(d_K, -m)_\infty\cdot\inv_\infty(B) = -1$, it follows that $\Diff_B(m)$ has odd
cardinality. If $\ell$ is a prime number split in $K$ then $\ell \nmid d_B$ by assumption and
$$
\QQ_\ell(\sqrt{d_K}) \cong K \otimes_\QQ \QQ_\ell \cong \QQ_\ell \times \QQ_\ell,
$$
so $-m$ is a norm from $\QQ_\ell(\sqrt{d_K})$
and thus $(d_K, -m)_\ell = 1$. Hence $(d_K, -m)_\ell\cdot\inv_\ell(B) = 1$, which shows $\ell \notin \Diff_B(m)$ if $\ell$
is split in $K$.

\begin{proposition}\label{diff}
Let $\frakp \subset \Oo_K$ be a prime lying over a prime $p$. If $\YYY_m(\overline{\FF}_\frakp) \neq \varnothing$ then
$\Diff_B(m) = \{p\}$.
\end{proposition}

\begin{proof}
Fix $(A, \kappa, f) \in \YYY_m(\overline{\FF}_\frakp)$. View $K$ as a $\QQ$-subalgebra of $B^{(p)}$ via $\kappa : K \to B^{(p)}$
and consider the element $f + f^t \in B^{(p)}$, where $f^t$ is the dual isogeny to $f$ (see \cite[\S 2]{P}).
By definition, $f^t = \lambda^{-1} \circ f^\vee \circ \lambda$, where $\lambda : A \to A^\vee$
is the canonical principal polarization on a QM abelian surface, so $f^t = f^\dagger$ where $g \mapsto g^\dagger$ is the Rosati involution on $\End_{\Oo_B}^0(A)$
corresponding to $\lambda$. Since $f + f^t$ is fixed by the Rosati involution, which corresponds to the main involution on the definite quaternion algebra $B^{(p)}$, we have $f + f^t \in \ZZ \subset \End_{\Oo_B}(A)$.
However, as $f$ is a special endomorphism, for any $x \in K$, 
$$
x(f + f^t) = xf + (\overline{x})^tf^t = f\overline{x} + (xf)^t = (f + f^t)\overline{x},
$$
where we are using $x^t = \overline{x}$ (\cite[Lemma 3.3]{P}), so from $f + f^t \in \ZZ$, it follows that $f + f^t = 0$. Hence
$$
m = \deg^\ast(f) = f \circ f^t = -f^2.
$$
Setting $\delta = \sqrt{d_K} \in K \subset B^{(p)}$, the $\QQ$-algebra $B^{(p)}$ is generated by elements $\delta, f$ satisfying
$$
\delta^2 = d_K, \quad f^2 = -m, \quad \delta f = -f\delta,
$$
the last relation coming from $\overline{\delta} = -\delta$, so
$$
B^{(p)} \cong \left(\frac{d_K, -m}{\QQ}\right).
$$
Therefore
$$
(d_K, -m)_\ell\cdot\inv_\ell(B) = \inv_\ell(B^{(p)})\cdot\inv_\ell(B) = \left\{\begin{array}{ll}
1 & \text{if $\ell \neq p, \infty$} \\
-1 & \text{if $\ell = p, \infty$},
\end{array} \right.
$$
which means $\Diff_B(m) = \{p\}$. 
\end{proof}

\begin{corollary}
If $\Diff_B(m) = \{p\}$ then there is a unique prime ideal $\frakp \subset \Oo_K$ over $p$ and $\YYY_m(\overline{\FF}_\frakq) = \varnothing$
for every prime $\frakq \neq \frakp$. If $\#\Diff_B(m) > 1$ then $\YYY_m = \varnothing$.
\end{corollary}

\begin{proof}
If $\YYY_m(\overline{\FF}_\frakq) \neq \varnothing$ then $\Diff_B(m) = \{q\}$ where $q\ZZ = \frakq \cap \ZZ$. Hence
$p = q$ and then $\frakp = \frakq$ since $p$ and $q$ are nonsplit in $K$.
\end{proof}

\section{Local quadratic spaces}
Let $m$ be a positive integer, $p$ a prime nonsplit in $K$, $\frakp \subset \Oo_K$ the prime over $p$, and suppose $(A, \kappa) \in \YYY(\overline{\FF}_\frakp)$ is supersingular. For each prime $\ell$ set 
$$
L_\ell(A, \kappa) = L(A, \kappa) \otimes_\ZZ \ZZ_\ell, \quad V_\ell(A, \kappa) = V(A, \kappa) \otimes_\QQ \QQ_\ell.
$$

\begin{proposition}\label{quadratic  III}
If $\ell \neq p$ is a prime then there is an $\Oo_{K, \ell}$-linear isomorphism of $\ZZ_\ell$-quadratic spaces
$$
(\Oo_{K, \ell}, \beta_\ell\cdot\NN_{K_\ell/\QQ_\ell}) \cong (L_\ell(A, \kappa), \deg^\ast)
$$
for some $\beta_\ell \in \ZZ_\ell$ with $\beta_\ell = -1$ if $\ell \nmid d_B$ and $\ord_\ell(\beta_\ell) = 1$ if $\ell \mid d_B$.
\end{proposition}

\begin{proof}
First suppose $\ell \nmid d_B$ and let $T_\ell = T_\ell(A)$ be the $\ell$-adic Tate module of $A$. The
standard idempotents $\varepsilon, \varepsilon' \in \MM_2(\ZZ_\ell) \cong \Oo_B \otimes_\ZZ \ZZ_\ell$ induce a decomposition
$T_\ell = \varepsilon T_\ell \oplus \varepsilon' T_\ell$. As the $\Oo_{K, \ell}$ and $\Oo_{B, \ell}$ actions on $T_\ell$ commute,
the $\ZZ_\ell$-module $\varepsilon T_\ell$ is an $\Oo_{K, \ell}$-module, free of rank 1 by considering $K_\ell$-dimensions.
There are $\ZZ_\ell$-algebra isomorphisms (see \cite[\S 3.4]{P})
$$
\End_{\Oo_B}(A) \otimes_\ZZ \ZZ_\ell \cong \End_{\Oo_B}(T_\ell) \cong \End_{\ZZ_\ell}(\varepsilon T_\ell) \cong \End_{\ZZ_\ell}(
\Oo_{K, \ell}).
$$
Let $f_0 \in \End_{\ZZ_\ell}(\Oo_{K, \ell})$ be defined by $f_0(x) = \overline{x}$. Then
$$
\End_{\ZZ_\ell}(\Oo_{K, \ell}) = \Oo_{K, \ell} \oplus \Oo_{K, \ell}\cdot f_0
$$
and $L_\ell(A, \kappa) = \Oo_{K, \ell}\cdot f_0$, so for any $xf_0 \in L_\ell(A, \kappa)$, 
$$
\deg^\ast(xf_0) = -(xf_0)^2 = -x\overline{x}f_0^2 = -\NN_{K_\ell/\QQ_\ell}(x).
$$
Therefore the map $\Oo_{K, \ell} \to L_\ell(A, \kappa)$ given by $x \mapsto xf_0$ defines an
$\Oo_{K, \ell}$-linear isomorphism of $\ZZ_\ell$-quadratic spaces
$$
(\Oo_{K, \ell}, -\NN_{K_\ell/\QQ_\ell}) \to (L_\ell(A, \kappa), \deg^\ast).
$$

Now suppose $\ell \mid d_B$. Viewing $K$ as a $\QQ$-subalgebra of $B^{(p)}$ via $\kappa$, there is a decomposition
$$
B^{(p)}_\ell = K_\ell \oplus K_\ell\cdot f_0
$$
for any $f_0 \in V_\ell(A, \kappa)$. Choosing $f_0$ to be an $\Oo_{K, \ell}$-generator of $L_\ell(A, \kappa)$, the map
$x \mapsto xf_0$ defines an isomorphism of $\ZZ_\ell$-quadratic spaces
$$
(\Oo_{K, \ell}, \beta_\ell\cdot\NN_{K_\ell/\QQ_\ell}) \to (L_\ell(A, \kappa), \deg^\ast)
$$
with $\beta_\ell = -f_0^2 = \deg^\ast(f_0)$. Then from
$$
B^{(p)}_\ell \cong \left(\frac{d_K, -\beta_\ell}{\QQ_\ell}\right)
$$
we have $(d_K, -\beta_\ell)_\ell = -1$ as $\ell \mid \disc(B^{(p)})$. 

There is an isomorphism of $\ZZ_\ell$-algebras $\End_{\Oo_B}(A) \otimes_\ZZ \ZZ_\ell \cong \Oo_{B, \ell}$ (see \cite[\S 3.4]{P}), which
is the unique maximal order in $B^{(p)}_\ell$, and the quadratic form $\deg^\ast$ on $\End_{\Oo_B}(A) \otimes_\ZZ \ZZ_\ell$
corresponds to the quadratic form of reduced norm on $\Oo_{B, \ell}$, so $f \in B^{(p)}_\ell$ is in $\End_{\Oo_B}(A) \otimes_\ZZ \ZZ_\ell$
if and only if $\deg^\ast(f) \in \ZZ_\ell$. As $(d_K, -\beta_\ell)_\ell = -1$, the element $-\beta_\ell \in \ZZ_\ell$ is not a norm
from $\QQ_\ell(\sqrt{d_K}) \cong K_\ell$, which means $\ord_\ell(-\beta_\ell) = \ord_\ell(\beta_\ell)$ is odd (since $K_\ell/\QQ_\ell$ is
unramified). If $\ord_\ell(\beta_\ell) \gqq 3$ then $\deg^\ast(\ell^{-1}f_0) \in \ZZ_\ell$ since $\deg^\ast(\ell) = \ell^2$, so
$\ell^{-1}f_0 \in L_\ell(A, \kappa)$. But $f_0$ is an $\Oo_{K, \ell}$-module generator of $L_\ell(A, \kappa)$, so this is a contradiction
and hence $\ord_\ell(\beta_\ell) = 1$. 
\end{proof}

\begin{proposition}\label{quadratic IV}
There is an $\Oo_{K, p}$-linear isomorphism of $\ZZ_p$-quadratic spaces
$$
(\Oo_{K, p}, \beta_p\cdot \NN_{K_p/\QQ_p}) \cong (L_p(A, \kappa), \deg^\ast)
$$
for some $\beta_p \in \ZZ_p$ satisfying $\ord_p(\beta_p) = 2 - e_p\varepsilon_p$, where $\varepsilon_p = 1 - \ord_p(d_B)$.
\end{proposition}

\begin{proof}
There is an $\Oo_{K, p}$-linear isomorphism of $\ZZ_p$-quadratic spaces
$$
(\Oo_{K, p}, \beta_p\cdot \NN_{K_p/\QQ_p}) \to (L_p(A, \kappa), \deg^\ast)
$$
given by $x \mapsto xf_0$, where $f_0$ is an $\Oo_{K, p}$-module generator of $L_p(A, \kappa)$ and $\beta_p = \deg^\ast(f_0)$.
First suppose $p \nmid d_B$. Then
$$
B_p^{(p)} \cong \left(\frac{d_K, -\beta_p}{\QQ_p}\right)
$$
implies $(d_K, -\beta_p)_p = -1$, and $\End_{\Oo_B}(A) \otimes_{\ZZ} \ZZ_p \cong \Delta$ is the unique maximal order in $B_p^{(p)}$ (\cite[\S 3.4]{P}).
Suppose $p$ is unramified in $K$, so $\ord_p(\beta_p)$ is odd. If $\ord_p(\beta_p) \gqq 3$ then $\deg^\ast(p^{-1}f_0) \in \ZZ_p$,
which means $p^{-1}f_0 \in L_p(A, \kappa)$. This is a contradiction, so $\ord_p(\beta_p) = 1$. Next suppose $p$ is ramified in
$K$ and let $\pi \in \Oo_{K, p}$ be a uniformizer. If $\ord_p(\beta_p) > 0$ then $\deg^\ast(\pi^{-1}f_0) \in \ZZ_p$ as 
$\NN_{K_p/\QQ_p}(\pi)$ is a uniformizer of $\ZZ_p$. Again this implies $\pi^{-1}f_0 \in L_p(A, \kappa)$, which is a contradiction,
so $\ord_p(\beta_p) = 0$.

Now suppose $p \mid d_B$. Then $\End_{\Oo_B}(A) \otimes_{\ZZ} \ZZ_p \cong R$, with
$$
R = \left\{\begin{bmatrix} x & y\Pi \\ py\Pi & x \end{bmatrix} : x, y \in \Oo_{K, p}\right\} \subset \MM_2(\Delta),
$$
where $\Pi \in \Delta$ is a uniformizer satisfying $\Pi x = \overline{x} \Pi$ for all $x \in \Oo_{K, p}$, 
and $\kappa : \Oo_{K, p} \to R$ is given by $\kappa(x) = \text{diag}(x, x)$
(see \cite[Proposition 3.30]{P}). It follows that $L_p(A, \kappa) = \Oo_{K, p} \cdot f_0$,
where
$$
f_0 = \begin{bmatrix} 0 & \Pi \\ p\Pi & 0 \end{bmatrix}.
$$
Since $\beta_p = \deg^\ast(f_0) = -p^2$ (\cite[Proposition 5.4]{P}), we have $\ord_p(\beta_p) = 2$.
\end{proof}

\section{Counting geometric points}
Define two algebraic groups $T$ and $T^1$ over $\QQ$ whose functors of points are given by
\begin{align*}
&T(R) = (K \otimes_{\QQ} R)^\times \\
&T^1(R) = \{x \in T(R) : \NN_{K/\QQ}(x) = 1\}
\end{align*}
for any $\QQ$-algebra $R$. Define a homomorphism $\eta : T \to T^1$ given on points by $\eta(x) = \overline{x}^{-1}x$. Let
$U = \widehat{\Oo}_K^\times \subset T(\widehat{\QQ}) = \widehat{K}^\times$, so $U = \prod_\ell U_\ell$ for some groups $U_\ell \subset
T(\QQ_\ell)$, and let $U^1 = \eta(U) = \prod_\ell U_\ell^1$ for some groups $U_\ell^1$. If $R$ is a field of characteristic $0$ or $\widehat{\QQ}$, then the sequence
\begin{equation}\label{exact sequence}
1 \to R^\times \to T(R) \mapp{\eta} T^1(R) \to 1
\end{equation}
is exact (\cite[Proposition 2.13]{Howard}), so in particular there is an isomorphism of groups
\begin{equation}\label{iso I}
T(\QQ)\dc T(\widehat{\QQ})/U \cong T^1(\QQ)\dc T^1(\widehat{\QQ})/U^1.
\end{equation}
Also, there is an isomorphism of groups
\begin{equation}\label{iso II}
T(\QQ)\dc T(\widehat{\QQ})/U \to \Cl(\Oo_K)
\end{equation}
given by 
$$
t \mapsto \prod_{\frakp \subset \Oo_K}\frakp^{\ord_{\frakp}(t_\frakp)}.
$$

Let $p$ be a prime that is nonsplit in $K$, let $\frakp \subset \Oo_K$ be the prime over $p$, and let $(A, \kappa) \in 
\YYY(\overline{\FF}_\frakp)$. Recall that $K$ acts on $V(A, \kappa)$ by $x \cdot f = \kappa(x) \circ f$. By restriction, the group
$T^1(\QQ) \subset K^\times$ acts on $V(A, \kappa)$, and for any $m \in \QQ^\times$, the set
$$
\{f \in V(A, \kappa) : \deg^\ast(f) = m\}
$$
is either empty or a simply transitive $T^1(\QQ)$-set. By composing with the homomorphism $\eta : T \to T^1$, the group $T(\QQ)$
acts on $V(A, \kappa)$, and this action is given by
$$
t \bullet f = \kappa(t) \circ f \circ \kappa(t)^{-1}.
$$

Now fix $t \in T(\widehat{\QQ})$ and let $\fraka \in \Cl(\Oo_K)$ be its image under (\ref{iso II}). We will write $\fraka \otimes A$ for
the CMQM abelian surface $\fraka \otimes_{\Oo_K} A$ determined by $(\fraka \otimes_{\Oo_K} A)(X) = \fraka \otimes_{\Oo_K} A(X)$ for any $\overline{\FF}_\frakp$-scheme $X$. There is an $\Oo_K$-linear quasi-isogeny
$f \in \Hom_{\Oo_B}(A, \fraka \otimes A) \otimes_{\ZZ} \QQ$,
given on points by $f(x) = 1 \otimes x$. Then the map
$\End^0_{\Oo_B}(\fraka \otimes A) \to \End^0_{\Oo_B}(A)$
given by $\varphi \mapsto f^{-1} \circ \varphi \circ f$ is an isomorphism of $K$-vector spaces, and restricting gives an isomorphism
$V(\fraka \otimes A, \kappa) \to V(A, \kappa)$. This map identifies $\End_{\Oo_B}(\fraka \otimes A)$ with the $\Oo_K$-submodule
$$
\kappa(\fraka) \circ \End_{\Oo_B}(A) \circ \kappa(\fraka^{-1}) \subset \End^0_{\Oo_B}(A)
$$
by \cite[Lemma 7.14]{Conrad} and identifies $L(\fraka \otimes A, \kappa)$ with $\kappa(\fraka) \circ L(A, \kappa) \circ \kappa(\fraka^{-1})$. Therefore there is 
a $\widehat{K}$-linear isomorphism
$\widehat{V}(A, \kappa) \cong \widehat{V}(\fraka \otimes A, \kappa)$
with $\widehat{L}(\fraka \otimes A, \kappa)$ isomorphic to the $\widehat{\Oo}_K$-submodule
$$
t \bullet \widehat{L}(A, \kappa) = \{\kappa(t) \circ f \circ \kappa(t)^{-1} : f \in \widehat{L}(A, \kappa)\}
$$
of $\widehat{V}(A, \kappa)$.

\begin{definition}
Let $(A, \kappa) \in \YYY(\overline{\FF}_\frakp)$. For each prime number $\ell$ and $m \in \QQ^\times$, define the \textit{orbital
integral} at $\ell$ by
$$
O_\ell(m, A, \kappa) = \sum_{t \in \QQ_\ell^\times\dc T(\QQ_\ell)/U_\ell}\mathbf{1}_{L_\ell(A, \kappa)}(t^{-1}\bullet f)
$$
if there is an $f \in V_\ell(A, \kappa)$ satisfying $\deg^\ast(f) = m$. If no such $f$ exists, set $O_\ell(m, A, \kappa) = 0$. 
\end{definition}

Here $\mathbf{1}_X$ is the characteristic function of a set $X$.
This definition does not depend on the choice of $f \in V_\ell(A, \kappa)$ such that $\deg^\ast(f) = m$ since $T(\QQ_\ell)$
acts simply transitively on the set of all such $f$.

\begin{proposition}\label{orbital III}
Let $p$ be a prime nonsplit in $K$, let $\frakp \subset \Oo_K$ be the prime over $p$, and suppose $(A, \kappa) \in \YYY(\overline{\FF}_\frakp)$.
For any $m \in \QQ^\times$ positive,
$$
\sum_{\fraka \in \Cl(\Oo_K)}\#\{f \in L(\fraka \otimes A, \kappa) : \deg^\ast(f) = m\} = \frac{|\Oo_K^\times|}{2}\prod_\ell O_\ell(m, A, \kappa).
$$
\end{proposition}

\begin{proof}
Using the isomorphisms (\ref{iso II}) and (\ref{iso I}),
$$
\sum_{\fraka \in \Cl(\Oo_K)}\#\{f \in L(\fraka \otimes A, \kappa) : \deg^\ast(f) = m\} =  \sum_{t \in T^1(\QQ)\dc T^1(\widehat{\QQ})/U^1}\sum_{\substack{f \in V(A, \kappa) \\ \deg^\ast(f) = m}}\mathbf{1}_{t\bullet\widehat{L}
(A, \kappa)}(f).
$$
Suppose there is an $f_0 \in V(A, \kappa)$ such that $\deg^\ast(f) = m$. Since the action of $T^1(\QQ)$ on the set of all such $f_0$
is simply transitive,
\begin{align*}
\sum_{t \in T^1(\QQ)\dc T^1(\widehat{\QQ})/U^1}\sum_{\substack{f \in V(A, \kappa) \\ \deg^\ast(f) = m}}\mathbf{1}_{t\bullet\widehat{L}
(A, \kappa)}(f)  &= \sum_{t \in T^1(\QQ)\dc T^1(\widehat{\QQ})/U^1}\sum_{\gamma \in T^1(\QQ)}\mathbf{1}_{t\bullet\widehat{L}
(A, \kappa)}(\gamma^{-1}\bullet f_0) \\
&= \sum_{t \in T^1(\QQ)\dc T^1(\widehat{\QQ})/U^1}\sum_{\gamma \in T^1(\QQ)}\mathbf{1}_{\gamma t\bullet\widehat{L}
(A, \kappa)}(f_0) \\
&= |T^1(\QQ) \cap U^1| \sum_{t \in T^1(\widehat{\QQ})/U^1} \mathbf{1}_{t \bullet \widehat{L}(A, \kappa)}(f_0) \\
&= \frac{|\Oo_K^\times|}{2} \prod_\ell O_\ell(m, A, \kappa),
\end{align*}
where we are using 
$$
T^1(\QQ) \cap U^1 \cong (T(\QQ) \cap U)/\{\pm 1\} = \Oo_K^\times/\{\pm 1\}
$$
and the isomorphism 
$$
\QQ_\ell^\times\dc T(\QQ_\ell)/U_\ell \cong T^1(\QQ_\ell)/U^1_\ell
$$
coming from the exact sequence (\ref{exact sequence}). If there is no such $f_0$ then by the Hasse-Minkowski theorem there
is some prime $\ell < \infty$ such that $(V_\ell(A, \kappa), \deg^\ast)$ does not represent $m$, since $V_\infty(A, \kappa)$ does represent
$m$. Thus $O_\ell(m, A, \kappa) = 0$ and both sides of the stated equality are $0$. 
\end{proof}

\begin{proposition}\label{orbital IV}
If $(A, \kappa)$ is any object of $\YYY(\overline{\FF}_\frakp)$ and $m$ is a positive integer, then
$$
\#[\YYY_m(\overline{\FF}_\frakp)] = 2^r\prod_\ell O_\ell(m, A, \kappa),
$$
where $r$ is the number of primes dividing $d_B$.
\end{proposition}

\begin{proof}
Since $\End_{\Oo_B \otimes_{\ZZ}\Oo_K}(A) \cong \Oo_K$, we have $\Aut(A, \kappa) \cong \Oo_K^\times$, so an element of
$\Aut(A, \kappa, f)$ is $\kappa(x)$ for some $x \in \Oo_K^\times$ satisfying $\kappa(x) \circ f = f \circ \kappa(x)$. But $f$ is a special endomorphism, which means $\kappa(x) = \kappa(\overline{x})$ and thus $x \in \{\pm 1\}$. This shows $\Aut(A, \kappa, f)
= \{\pm 1\}$ for $f \in L(A, \kappa)$.
As the group $W_B \times \Cl(\Oo_K)$ acts simply transitively on the set $[\YYY(\overline{\FF}_\frakp)]$,
\begin{align*}
\#[\YYY_m(\overline{\FF}_\frakp)] &= \sum_{(A, \kappa) \in [\YYY(\overline{\FF}_\frakp)]}\sum_{\substack{f \in V(A, \kappa) \\ \deg^\ast(f) = m
}}\frac{|\Aut(A, \kappa, f)|}{|\Aut(A, \kappa)|}\cdot\mathbf{1}_{\widehat{L}(A, \kappa)}(f) \\
&= \frac{2}{|\Oo_K^\times|}\sum_{g \in W_B \times \Cl(\Oo_K)} \sum_{\substack{f \in V(g \cdot A, \kappa) \\ \deg^\ast(f) = m}}
\mathbf{1}_{\widehat{L}(g \cdot A, \kappa)}(f).
\end{align*}
But the action of $W_B$ on $[\YYY(\overline{\FF}_\frakp)]$ does not change the underlying QM abelian surface or the CM action (see \cite[\S 3.2]{P}),
so there is an isomorphism $V(w\cdot A, \kappa) \cong V(A, \kappa)$ for any $w \in W_B$. Therefore
$$
\#[\YYY_m(\overline{\FF}_\frakp)] = \frac{2|W_B|}{|\Oo_K^\times|}\sum_{\fraka \in \Cl(\Oo_K)}\sum_{\substack{f \in V(\fraka \otimes A,
\kappa) \\ \deg^\ast(f) = m}}\mathbf{1}_{\widehat{L}(\fraka \otimes A, \kappa)}(f) 
= 2^r\prod_\ell O_\ell(m, A, \kappa)
$$
by Proposition \ref{orbital III}.
\end{proof}

Recall the definitions of the functions $\rho$ and $\rho_\ell$ from the introduction.

\begin{proposition}\label{orbital V}
Let $\ell$ be a prime, $m$ a positive integer, and $(A, \kappa) \in \YYY(\overline{\FF}_\frakp)$. 
If the quadratic space $(V_\ell(A, \kappa), \deg^\ast)$ represents $m$, then
$$
O_\ell(m, A, \kappa) = e_\ell \rho_\ell(md_B^{-1}p^{(e_p-1)\varepsilon_p - 1}).
$$
\end{proposition}

\begin{proof}
Fix an $f \in V_\ell(A, \kappa)$ satisfying $\deg^\ast(f) = m$ and fix an isomorphism
$$
(\Oo_{K, \ell}, \beta_\ell \cdot \NN_{K_\ell/\QQ_\ell}) \cong (L_\ell(A, \kappa), \deg^\ast)
$$
with $\beta_\ell$ as in Propositions \ref{quadratic III} and \ref{quadratic IV}. Using the isomorphism
$$
\QQ_\ell^\times\dc T(\QQ_\ell)/U_\ell \cong T^1(\QQ_\ell)/U^1_\ell
$$
we have
$$
O_\ell(m, A, \kappa) = \sum_{t \in T^1(\QQ_\ell)/U^1_\ell}\mathbf{1}_{\Oo_{K, \ell}}(t^{-1}f).
$$
First suppose $\ell$ is inert in $K$. Then $\QQ_\ell^\times\dc K_\ell^\times/U_\ell = \{1\}$, so
$T^1(\QQ_\ell)/U^1_\ell = \{1\}$. Hence
$$
O_\ell(m, A, \kappa) = \mathbf{1}_{\Oo_{K, \ell}}(f) = \rho_\ell(m\beta_\ell^{-1})
$$
since $\NN_{K_\ell/\QQ_\ell}(f) = m\beta_\ell^{-1}$. 
Next suppose $\ell$ is ramified in $K$ and let $\pi \in \Oo_{K, \ell}$ be a uniformizer. Then
$\QQ_\ell^\times\dc K_\ell^\times/U_\ell = \{1, \pi\}$ and $T^1(\QQ_\ell)/U^1_\ell = \{1, u\}$ where
$u = \overline{\pi}^{-1}\pi \in \Oo_{K, \ell}^\times$, so
$$
O_\ell(m, A, \kappa) = \mathbf{1}_{\Oo_{K, \ell}}(f) + \mathbf{1}_{\Oo_{K, \ell}}(u^{-1}f) = 2\rho_\ell(m\beta_\ell^{-1}).
$$
Finally suppose $\ell$ is split in $K$, so $K_\ell \cong \QQ_\ell \times \QQ_\ell$. Then
$$
\QQ_\ell^\times\dc K_\ell^\times/U_\ell = \{(\ell^k, 1) : k \in \ZZ\}
$$
and $T^1(\QQ_\ell)/U^1_\ell = \{(\ell^k, \ell^{-k}) : k \in \ZZ\}$. Writing $f = (f_1, f_2) \in \QQ_\ell \times \QQ_\ell$,
we have
\begin{align*}
O_\ell(m, A, \kappa) &= \sum_{k \in \ZZ}\mathbf{1}_{\ZZ_\ell \times \ZZ_\ell}(\ell^kf_1, \ell^{-k}f_2) \\
&= 1 + \ord_\ell(f_1f_2) \\
&= 1 + \ord_\ell(m\beta_\ell^{-1}) \\
&= \rho_\ell(m\beta_\ell^{-1}).  \qedhere
\end{align*}
\end{proof}

\begin{theorem}\label{geometric points II}
Let $m$ be a positive integer. If $\Diff_B(m) = \{p\}$ then
$$
\#[\YYY_m(\overline{\FF}_\frakp)] = 2^{r + s}\rho(md_B^{-1}p^{(e_p-1)\varepsilon_p - 1}),
$$
where $\frakp \subset \Oo_K$ is the unique prime over $p$ and $s$ is the number of distinct primes dividing $d_K$. Furthermore, the number $\#[\YYY_m(\overline{\FF}_\frakp)]$
is nonzero, unless $p \mid d_B$ and $\ord_p(m) = 0$.
\end{theorem}

\begin{proof}
Let $(A, \kappa) \in \YYY(\overline{\FF}_\frakp)$, so $\End^0_{\Oo_B}(A) \cong B^{(p)}$. From $\Diff_B(m) = \{p\}$ we have
$$
(d_K, -m)_\ell = \left\{\begin{array}{ll}
-1 & \text{if $\ell \mid \disc(B^{(p)})$} \\
1 & \text{if $\ell \nmid \disc(B^{(p)})$}
\end{array} \right.
$$
for any prime $\ell$, so there is an isomorphism
$$
B^{(p)} \cong \left(\frac{d_K, -m}{\QQ}\right).
$$
Hence $B^{(p)}$ has a $\QQ$-basis $\{1, \delta, f, \delta f\}$ satisfying
$$
\delta^2 = d_K, \quad f^2 = -m, \quad \delta f = -f\delta.
$$
Embed $K$ into $B^{(p)}$ via $\sqrt{d_K} \mapsto \delta$. Then $\{f, \delta f\}$ is a $\QQ$-basis for 
$V(A, \kappa) \subset \End^0_{\Oo_B}(A)$ and $\Nrd(f) = m$. Thus, there is an $f \in V(A, \kappa)$ satisfying 
$\deg^\ast(f) = m$. Then by Propositions \ref{orbital IV} and \ref{orbital V},
\begin{align*}
\#[\YYY_m(\overline{\FF}_\frakp)] &= 2^r\prod_\ell O_\ell(m, A, \kappa) \\
&= 2^r\prod_\ell e_\ell \rho_\ell(md_B^{-1}p^{(e_p-1)\varepsilon_p - 1}) \\
&= 2^{r+s}\rho(md_B^{-1}p^{(e_p-1)\varepsilon_p - 1}).
\end{align*}

Now we will show that this number is nonzero by showing $\rho_\ell = \rho_\ell(md_B^{-1}p^{(e_p-1)\varepsilon_p - 1})$
is nonzero for each prime $\ell$.  If $\ell \neq p$ and $\ell \nmid d_B$,
then $(d_K, -m)_\ell = 1$, which means $-m \in \NN_{K_\ell/\QQ_\ell}(K_\ell)$ and thus
$\rho_\ell = \rho_\ell(m) > 0$. The other cases are similar except when $p \mid d_B$. In this case $(d_K, -m)_p = 1$, which implies $\ord_p(m)$ is even and therefore $\rho_p = \rho_p(mp^{-2}) > 0$, unless $\ord_p(m) = 0$.
\end{proof}

\section{Deformation theory and final formula}
Fix a  prime $p$ nonsplit in $K$ and let $\frakp \subset \Oo_K$ be the prime over $p$. Let $\WWW$
be the ring of integers in the completion of the maximal unramified extension of $K_\frakp$, so $\WWW$
is an $\Oo_K$-algebra. Let $\mathbf{CLN}$ be the category of complete local Noetherian $\WWW$-algebras
with residue field $\overline{\FF}_\frakp$, where a morphism $R \to R'$ is a local  $\WWW$-algebra homomorphism inducing the
identity $\overline{\FF}_\frakp \to \overline{\FF}_\frakp$ on residue fields. 

For $x = (A, i, \kappa) \in \YYY(\overline{\FF}_\frakp)$ define a functor $\Def_{\Oo_B}(A, \Oo_K) : \mathbf{CLN} \to \mathbf{Sets}$ by assigning to each $R \in \mathbf{CLN}$ the set of isomorphism
classes of deformations of $x$ to $R$. By \cite[Proposition 3.10]{P}, $\Def_{\Oo_B}(A, \Oo_K)$ is represented
by $\WWW$. 
For $(A, i, \kappa) \in \YYY(\overline{\FF}_\frakp)$ and $f \in \End_{\Oo_B}(A)$, define a functor
$\Def(A, \kappa, f) : \mathbf{CLN} \to \mathbf{Sets}$ by assigning to each $R$
the set of isomorphism classes of deformations of $(A, i, \kappa, f)$ to $R$.
If $R \in \mathbf{CLN}$, $x = (A, i, \kappa, f) \in \YYY_m(\overline{\FF}_\frakp)$, and
$\widetilde{x} = (\widetilde{A}, \widetilde{i}, \widetilde{\kappa}, \widetilde{f})$ is a deformation of $x$ to $R$, then we must have $\widetilde{x} \in \YYY_m(R)$. 

Now fix a positive integer $m$ and a  triple $(A, \kappa, f) \in \YYY_m(\overline{\FF}_\frakp)$.
Let $\frakg$ be the connected $p$-divisible group of height $2$ and dimension $1$ over $\overline{\FF}_\frakp$.

\begin{proposition}\label{representability VI}
If $p \mid d_B$ then $\Def(A, \kappa, f)$ is represented by a local Artinian $\WWW$-algebra of
length $\frac{1}{2}\ord_p(m)$.
\end{proposition}

\begin{proof}
Since $p$ is inert in $K$, $\WWW = W$ is the ring of integers in the completion of the maximal unramified extension of $\QQ_p$ (the usual $p$-Witt ring of $\overline{\FF}_p$). Fix a uniformizer $\Pi \in \Delta$ satisfying $\Pi x = x^\iota\Pi$ for all $x \in \Oo_p \subset \Delta$, where $\iota$ is the main involution on $\Delta_{\QQ}$ and
$\Oo_p$ is the image of the CM action $\Oo_{K, p} \to \Delta$ on an elliptic curve $E$ such that 
$A \cong M \otimes_{\Oo_K} E$.
Then there is an isomorphism of $\ZZ_p$-algebras $\End_{\Oo_B}(A) \otimes_{\ZZ} \ZZ_p \cong R$, where
$$
R = \left\{\begin{bmatrix} x & y\Pi \\ py\Pi & x \end{bmatrix} : x, y \in \Oo_p \right\},
$$
so there is a decomposition of left $\Oo_p$-modules $R = \Oo_p \oplus \Oo_pP$, with the first factor embedded diagonally
and
$$
P = \begin{bmatrix} 0 & \Pi \\ p\Pi & 0 \end{bmatrix}.
$$
It follows that $L_p(A, \kappa) = \Oo_pP$ and hence for any integer $n \gqq 1$,
\begin{align*}
f \in \Oo_p + p^{n-1}R &\iff f \in p^{n-1}\Oo_pP \\
&\iff \ord_p(\deg^\ast(f)) \gqq 2n  \\
&\iff \tfrac{1}{2}\ord_p(m)  \gqq n.
\end{align*}

It follows from \cite[Proposition 2.9]{RZ} that the functor $\Def(A, \kappa, f)$ is represented by $W_n = W/(p^n)$ where $n$ is the largest integer such that
$f \in \End_{\Oo_B}(A[p^{\infty}]) \cong R$ lifts to an element of 
$\End_{\Oo_B}(\widetilde{A}[p^{\infty}] \otimes_W W_n)$, where $\widetilde{A}$ is the universal deformation of $(A, i, \kappa)$ to $W$. By \cite[Lemma 6.3]{P},
$$
\End_{\Oo_B}(\widetilde{A}[p^{\infty}] \otimes_W W_n) \cong \Oo_p + p^{n-1}R,
$$
so the result follows from the above calculation since $\text{length}(W/(p^n)) = \text{length}_W(W/(p^n)) = n$.
\end{proof}

\begin{theorem}\label{local ring III}
Suppose $p$ is a prime nonsplit in $K$, let $\frakp \subset \Oo_K$ be the prime over $p$, and let
$m \in \ZZ^+$. For any $y \in \YYY_m(\overline{\FF}_\frakp)$, the strictly Henselian local ring 
$\OOO^{\text{\upshape sh}}_{\YYY_m, y}$ is Artinian of length
$$
\varepsilon_p + e_p\frac{\ord_p(md_K) - \varepsilon_p}{2}.
$$
\end{theorem}

\begin{proof}
The same proof as in \cite[Proposition 2.25]{Howard} shows that the functor $\Def(A, \kappa, f)$
is represented by the ring $\widehat{\OOO}^{\text{sh}}_{\YYY_m, y}$, where $y = (A, \kappa, f) \in \YYY_m(\overline{\FF}_\frakp)$, so the result for $p \mid d_B$ follows from Proposition 
\ref{representability VI}. The idea for the $p \nmid d_B$ case is to reduce it to the analogous result for elliptic curves as follows.

Fix $y = (A, \kappa, f) \in \YYY_m(\overline{\FF}_\frakp)$ for $p \nmid d_B$. Then the standard
idempotents $\varepsilon, \varepsilon' \in \MM_2(\WWW) \cong \Oo_B \otimes_{\ZZ} \WWW$
induce a splitting
$$
A[p^{\infty}] \cong \varepsilon A[p^{\infty}] \times \varepsilon'A[p^{\infty}] \cong \frakg \times \frakg,
$$
where $\Oo_B$ acts through the natural action of $\MM_2(\WWW)$. Also, if $\Oo_p = \kappa(\Oo_{K, p}) \subset \Delta \cong \End_{\overline{\FF}_\frakp}(\frakg)$, the action of $\Oo_K$ on $A[p^{\infty}]$ is through the diagonal action of $\Oo_p$.
By the Serre-Tate theorem there is an isomorphism of functors
$
\Def(A, \kappa, f) \cong \Def(A[p^{\infty}], \kappa[p^{\infty}], f[p^{\infty}]),
$
where the functor on the right assigns to each $R \in \mathbf{CLN}$ the set of isomorphism classes of deformations of $A[p^{\infty}]$, with its actions of $\Oo_B$ and $\Oo_K$, and the endomorphism $f[p^{\infty}]$, to $R$. As in \cite[Proof of Proposition 5.7]{P}, 
there are isomorphisms
$$
\End_{\Oo_B}(A) \otimes_{\ZZ}\ZZ_p \cong \End_{\Oo_B}(A[p^{\infty}]) \cong \End_{\overline{\FF}_\frakp}(\frakg) \cong \Delta,
$$
identifying the quadratic form $\deg^\ast$ with the reduced norm $\Nrd$ on $\Delta$.
Let $f_0$ be the image of $f$ under these isomorphisms, so $f_0 \in \Delta \sm \Oo_p$, being a special endomorphism. 

Define a functor $\Def(\frakg, \Oo_p[f_0]) : \mathbf{CLN} \to \mathbf{Sets}$ in the obvious way. Then there is a natural isomorphism of functors
$$
\Def(\frakg, \Oo_p[f_0]) \to \Def(A[p^{\infty}], \kappa[p^{\infty}], f[p^{\infty}])
$$
given by $(\frakG, g) \mapsto (\frakG \times \frakG, \diag(g, g))$, where $\frakG$ is a deformation of $\frakg$, together with an $\Oo_p$-action, lifting the action on $\frakg$, and  $g$ is an endomorphism lifting $f_0$. On the right, $\Oo_K$ acts on $\frakG \times \frakG$ diagonally and $\Oo_B$ acts through $\MM_2(\WWW)$. That the above morphism is an isomorphism follows from the fact that both functors are represented by $\WWW_n = \WWW/(\pi^n)$, where $\pi \in \Oo_{K_{\frakp}}$ is a uniformizer, and $n$ is the largest integer such that $f_0$ lifts to an element of 
$$
\End_{\WWW_n}(\widetilde{\frakg} \otimes_{\WWW} \WWW_n) \cong \End_{\Oo_B}(\widetilde{A}[p^{\infty}] \otimes_{\WWW} \WWW_n),
$$
with $\widetilde{\frakg}$ the universal deformation of $\frakg$, with its $\Oo_p$-action, to $\WWW$, and $\widetilde{A}[p^{\infty}]$ the universal deformation of $A[p^{\infty}]$ with its $\Oo_B \otimes_{\ZZ} \Oo_p$-action, to $\WWW$.

Fix an isomorphism $\frakg \cong E[p^{\infty}]$ for some supersingular elliptic curve $E$ over $\overline{\FF}_\frakp$. View $\End_{\Oo_B}(A)$ as an order in $\End^0_{\Oo_B}(A[p^{\infty}]) \cong \Delta_{\QQ} \cong \End^0_{\overline{\FF}_\frakp}(E[p^{\infty}])$ via the natural inclusion (\cite[Lemma 3.2]{Conrad}) and the same for $\End_{\overline{\FF}_\frakp}(E) \hookrightarrow \End^0_{\overline{\FF}_\frakp}(E[p^{\infty}])$. Then 
$\End_{\overline{\FF}_\frakp}(E)$ is a maximal order, and replacing $E$ with an isogenous elliptic curve, we may assume $\End_{\overline{\FF}_\frakp}(E)$ contains $\End_{\Oo_B}(A)$ (\cite[Corollary 42.2.21]{Voight}). Hence, there is an $\Oo_K$-action $\kappa_0$ on $E$ and a special endomorphism $h \in 
\End_{\overline{\FF}_\frakp}(E)$ such that $h$ is sent to $f_0$ under the natural isomorphism
$\End_{\overline{\FF}_\frakp}(E) \otimes_{\ZZ}\ZZ_p \to \Delta$. This isomorphism identifies the quadratic form $\deg$ with $\Nrd$ (\cite[Proof of Lemma 2.11]{Howard}), so we also have $\deg(h) = m$, giving a geometric point $z = (E, \kappa_0, h)  \in \ZZZ_m(\overline{\FF}_{\frakp})$, with $\ZZZ_m$ the stack defined in Section 1.1.

Finally, by the Serre-Tate theorem again, there is a natural isomorphism of functors $$\Def(\frakg, \Oo_p[f_0]) \cong \Def(E, \kappa_0, h).$$ As above, the deformation functor $\Def(E, \kappa_0, h)$ is represented by the ring $\widehat{\OOO}^{\text{sh}}_{\ZZZ_m, z}$. Putting it all together, there is an isomorphism of rings $\widehat{\OOO}^{\text{sh}}_{\YYY_m, y} \cong
\widehat{\OOO}^{\text{sh}}_{\ZZZ_m, z}$, and this case of the theorem follows from a result of 
Gross giving the length of the latter ring (\cite[Theorem 5.11]{KRY}):
\begin{equation*}
\length(\widehat{\OOO}^{\text{sh}}_{\ZZZ_m, z}) = 1 + \frac{\ord_p(md_K/p)}{f_p}. \qedhere
\end{equation*}
\end{proof}

\begin{theorem}\label{final formula II}
Let $m \in \ZZ^+$ and suppose $\Diff_B(m) = \{p\}$. Then
$$
\deg(\YYY_m) = 2^{r+s}\log(p)\cdot \rho(md_B^{-1}p^{(e_p-1)\varepsilon_p - 1})\cdot (\ord_p(md_K) + \varepsilon_pf_p - \varepsilon_p).
$$
If $\#\Diff_B(m) > 1$ then $\deg(\YYY_m) = 0$.
\end{theorem}

\begin{proof}
Let $\frakp \subset \Oo_K$ be the prime over $p$. Since $\YYY_m(\overline{\FF}_\frakq) = \varnothing$
for all primes $\frakq \neq \frakp$, for any $y \in \YYY_m(\overline{\FF}_\frakp)$,
\begin{align*}
\deg(\YYY_m) &= \log(|\FF_\frakp|)\cdot\#[\YYY_m(\overline{\FF}_\frakp)]\cdot \text{length}(\OOO^{\text{sh}}_{\YYY_m, y}) \\
&= f_p\cdot \log(p)\cdot 2^{r+s}\rho(md_B^{-1}p^{(e_p-1)\varepsilon_p - 1})\cdot \left(\varepsilon_p + e_p\frac{\ord_p(md_K) - \varepsilon_p}{2}\right)  \\
&= 2^{r+s}\log(p)\cdot \rho(md_B^{-1}p^{(e_p-1)\varepsilon_p - 1})\cdot (\ord_p(md_K) + \varepsilon_pf_p - \varepsilon_p) 
\end{align*}
by Theorems \ref{geometric points II} and \ref{local ring III}. If $\#\Diff_B(m) > 1$ then $\YYY_m = \varnothing$.
\end{proof}

\end{document}